\documentclass[11pt]{amsart}

\usepackage{amsmath}
\usepackage{amssymb,amscd}
\usepackage{mathrsfs}       % for the script X used by \XXS
\usepackage{srcltx}
\usepackage[pdftex]{graphicx}
\usepackage{epsfig}
\usepackage{graphicx}

\numberwithin{equation}{section}
\newtheorem{theorem}{Theorem}[section]

\newtheorem{corollary}[theorem]{Corollary}

\newtheorem{lemma}[theorem]{Lemma}
\newtheorem{prop}[theorem]{Proposition}
\newtheorem{remark}[theorem]{Remark}

\theoremstyle{definition}
\newtheorem{defn}[theorem]{Definition}

\newtheorem{example}[theorem]{Example}

\def \begineq{\begin{equation}}
\def \endeq{\end{equation}}

\def \bb{\mathbb}

\def \mf{\mathfrak}

\def\zS{\mathbb S}    %%%%%%%%% Sphere
\def\oO{\mathcal O}    %%%%%%%%% Orbit

\newcommand{\abs}[1]{\lvert#1\rvert}  %%%% absolute value %%%%%

\def \CC{{\bb{C}}}

\def \RR{{\bb{R}}}
\def \TT{{\bb{T}}}

\def \ZZ{{\bb{Z}}}

\def \({\left(}
\def \){\right)}
\def \<{\langle}
\def \>{\rangle}

\def \qed{\hfill $\square$ \vspace{0.03in}}

\begin{document}

\title{Some Aspects of Equivariant LS-category}
 %MSC classification
 %53C15; 53D20
\author[M. Bayeh]{Marzieh Bayeh}

\address{Department of Mathematics and Statistics, University of Regina, Regina S4S0A2, Canada.}

\email{bayeh20m@uregina.ca}

\author[S. Sarkar]{Soumen Sarkar}

\address{Department of Mathematics and Statistics, University of Regina, Regina S4S0A2, Canada.}

\email{soumen.sarkar@uregina.ca}



\abstract 
We study the lower bounds and upper bounds for LS-category and equivariant LS-category. 
In particular we compute both invariants for torus manifolds. 
There are some examples to show the sharpness of conditions in the theorems. 
Moreover the equivariant LS-category of the product space is discussed and counterexamples of some previous results are given.
\endabstract

\maketitle

%**********************************************************************************
%***************************      1- Introduction        **************************
%**********************************************************************************
\section{Introduction}\label{introsco2}

Let $G$ be a compact, Hausdorff, topological group, acting on a Hausdorff topological space $X$. In most cases $G$ is a Lie group acting on a compact manifold $X$. The equivariant LS-category of $X$, denoted by $cat_G(X)$ was introduced by Marzantowicz in \cite{Mar}, as a generalization of classical category of a space \cite{LS}, which is called Lusternik-Schnirelmann category \cite{Ja}. Marzantowicz showed that for a compact Lie group $G$, classical $cat$ of orbit space is a lower bound for $cat_G$,
$$ cat(X/G) \leq cat_G(X).$$
Colman studied the $cat_G(X)$ for finite group $G$ in \cite{Co} and gave an upper bound in terms of the dimension of orbit space and $cat_G$ of the singular set for the action. 
In \cite{HT}, Hurder and T\"{o}ben proved that for a manifold $M$ with a proper $G$-action, where $G$ is a Lie group, the number of components of the fixed point set is a lower bound for $cat_G(M)$. Later $cat_G(X)$ is studied by Colman and Grant \cite{CG}, for a compact Hausdorff topological group $G$, acting continuously on a Hausdorff space $X$.

Similar to definition of classical $cat$, $cat_G(X)$ is defined to be the least number of open subsets of $X$, which form a covering for $X$ and each open subset is equivariantly contractible to an orbit, rather than a point (see Definition \ref{defGcat}).

In this paper we study LS-cat and equivariant LS-cat. We compute these two invariants for locally standard torus manifolds, which are even dimensional smooth manifolds with locally standard action by half-dimensional compact torus (see Definition \ref{qtd02}).  
In Section \ref{sec3}, we study $cat_G(X)$ in terms of fixed point set $X^G$ and $cat_G(X^G)$, and some lower and upper bounds for $cat_G(X)$ are given. Also we study the upper bound for equivariant LS-cat of product space. 
In Section \ref{sec2}, some results on locally standard torus manifolds as well as simply connectedness of torus manifolds are discussed. 
In Section \ref{sec4}, the classical $cat$ of quasitoric manifolds are computed. We show that the equivariant connected sum of quasitoric manifolds does not affect the value of classical $cat$, i.e. for $2n$-dimensional quasitoric manifolds $M_1$ and $M_2$,
$$ cat(M_1 \#_{\TT^k} M_2) = cat(M_1) = cat(M_2) = n+1 \; ,$$
for any $k,n$ except $k=n=2$. Besides we examine the situations that for $4$-dimensional locally standard torus manifold $M$, the equality holds, meaning $ cat(M) = 3 $, see Theorem \ref{catM3}. Moreover the explicit construction of categorical covering for $M$ is also given. The special technique which is used for the construction leads us to generalize the idea for computing LS-cat of locally standard torus manifold in case there exists a triangulation of the orbit space.
In Section \ref{sec5}, $cat_{\TT^n}$ of quasitoric manifolds, as well as their equivariant connected sum are computed. We also prove the inequality for equivariant LS-cat of product space. 
%Moreover we examined the exact value for $cat_{\TT^2}$ of 4-dimensional locally standard torus manifolds.
Moreover a lower and upper bound for $cat_{\TT^n}$ of $4$-dimensional locally standard torus manifold are given where lower bound is sharp, see Theorem \ref{qtm4}.
Section \ref{sec6} is dedicated to computation of equivariant LS-cat. There are two counterexamples relevant to the work of Colman and Grant \cite{CG} in the following way. In their paper there are two statements on $cat_G$ of product, one with diagonal action, Theorem 3.15, and another with product action, Theorem 3.16. However there the hypotheses are not sufficient and lead to the counterexamples (but the subsequent results
 in \cite{CG}, in particular Corollary 5.8, are unaffected). Finally the equivariant LS-category of lens spaces is computed.

%**********************************************************************************
%******************        3- Equivariant LS-category         *********************
%**********************************************************************************
\section{Equivariant LS-category}\label{sec3}

In this section we prove a number of results for $cat_G(X)$ in terms of the fixed point set $X^G$. We begin by recalling some definitions and fixing some notations. Let $G$ be a compact Hausdorff topological group, acting continuously on a Hausdorff topological space $X$. In this case $X$ is called a $G$-space. For each $x \in X$, the set 
$$\oO(x) = \{ g.x : \; g \in G\}$$
is called the orbit of $x$, and 
$$ G_x = \{g\in G : \; g.x = x\}$$
is called the isotropy group or stabilizer of $x$.
 The set $X/G$ of all equivalence classes determined by the action, and equipped with the quotient topology is called the orbit space.
The set 
$$X^G = \bigg\{ x \in X : \; \forall g \in G , \; g.x = x \; \bigg\}$$
is called the fixed point set of $X$. Here $X^G$ is endowed with subspace topology. We denote the closed interval $[0,1]$ in $\RR$ by $I$ and $I^0 = (0,1)$.
%%==============
\begin{defn}
Let $X$ be a topological space, and $G$ be a topological group acting on $X$. 
\begin{enumerate}
\item An open subset $U$ of $X$, is called $G$-open set (or $G$-invariant) if $U$ is stable under $G$-action; i.e. $GU\subseteq U$.
\item Let $U$ be a $G$-invariant subset of $X$, the homotopy $H:U \times I \to X$ is called $G$-homotopy, if for every $g \in G$, $x \in U$, and $t \in I$,
$$ gH(x,t) = H(gx,t).$$
\item Let $U$ be a $G$-invariant subset of $X$, then $U$ is called $G$-categorical if there exists a $G$-homotopy $H:U \times I \to X$ such that $H(x,0) = x$ for each $x \in U$, and $H(U,1)$ is a subset of an orbit.
\end{enumerate}
\end{defn}
%==============
\begin{defn}\label{defGcat}
A $G$-categorical covering for a $G$-space $X$ is a finite number of $G$-categorical subsets $\{U_i\}_{i=1}^{n}$ that form a covering for $X$. The least value of $n$ for which such a covering exists, is called the equivariant category of $X$, denoted $cat_G(X)$. If no such covering exist, we write $cat_G(X) = \infty$.
\end{defn}
%==============
%==============
\begin{lemma}\label{GcatFix}
Let $U$ be a $G$-categorical subset of $G$-space $X$, which contains a fixed point $x_0 \in X^G$. Then $U$ is equivariantly contractible to $x_0$. In this case $U$ is called $G$-contractible, and denoted by $U\simeq_{G} x_0$.
\end{lemma}
\begin{proof}
Let $H:U \times I \to X$ be a $G$-homotopy, where $H(x,0)=x$ , $H(x,1) \in \oO(z)$ for some $z\in X$. Since 
$ gH(x_0,t) = H(gx_0,t) = H(x_0,t)$, it is easy to see that for all $t\in I$, $H(x_0 ,t) \in X^G$.
Therefore $H(x_0,1) \in X^G$, which implies $ \oO(z) = \{H(x_0,1)\}$.
Define $H' : U \times I \to X$, by 
$$ H'(x,t) = \left\{ 
\begin{array}{lll}
H(x,2t) & : & 0\leq t \leq \frac{1}{2} \\
 & &\\
H(x_0, 2-2t) & : & \frac{1}{2} \leq t \leq 1.
\end{array}
\right. $$
Clearly $H'$ is a $G$-homotopy. The lemma follows.
\end{proof}
~\\
%==============
Note that for a $G$-categorical set $U$, which contains a fixed point $x_0$,
 there exists a path $\Phi : I \to X^G$, defined by $\Phi(t) =H(x_0,t)$.
~\\
%==============
\begin{defn}
$x_0 \in X^G$ is called an isolated fixed point if there exists a neighborhood $U$ of $x_0$ that does not contain any other fixed points.
\end{defn}
%==============
\begin{lemma}\label{isolatpt}
Let $X$ be a Hausdorff space, and $U$ be a $G$-categorical subset that contains an isolated fixed point $x_0$. Then the $G$-homotopy $H:U \times I \to X$ fixes $x_0$, and $x_0$ is the only fixed point of $U$.
\end{lemma}
\begin{proof}
Let $V$ be an open neighborhood of $x_0$ that does not contain any other fixed points, and $\Phi :I \to X^G$ where $\Phi(t)=H(x_0,t)$. The set $\{x_0\} = V \cap X^G$ is open in $X^G$, and also closed (since $X^G$ is Hausdorff). Therefore the set $\{x_0\}$ is a path-connected component of $X^G$.
Thus for all $t \in I$, $\Phi(t) = x_0$ and hence $H$ fixes $x_0$.

\end{proof}
%==============
\begin{corollary}
If $X^G \neq \emptyset $ and $cat_G(X) = 1$, then $X$ is $G$-contractible to a point.
\end{corollary}
%==============
Note that in general case if $cat_G(X) =1$, $X$ may not be necessarily contractible. As for $G=\zS^1$, which acts on $X=\zS^1$, by product, $cat_G(X) =1$, while $X$ is not contractible.
%==============
\begin{lemma}
Let $(X,x_0)$ and $(Y,y_0)$ be pointed $G$-spaces. By pointed $G$-space, it means a $G$-space with base point such that the base point is fixed by $G$. Then
$$ cat_G(X \vee Y) \leq cat_G(X) + cat_G(Y) -1. $$
\end{lemma}
\begin{proof}
Let $\{U_i\}_{i=1}^{n}$ and $\{V_j\}_{j=1}^{m}$ be $G$-categorical covering for $X$ and $Y$ respectively. Let $x_0 \in U_i$ and $y_0 \in V_j$ for some $i$ and $j$. By Lemma \ref{GcatFix} $U_i \simeq_G x_0$ and $V_j \simeq_G y_0$. By identifying $x_0 = y_0$, one can show that $U_i \cup V_j$ is $G$-contractible to $x_0$ in $X \vee Y$.
\end{proof}
%==============
\begin{lemma}
Let $U$ be a $G$-categorical subset in $X$.
If  $U' = U \cap X^G \neq \emptyset$, then $U'$ is a $G$-categorical subset in $X^G$.
\end{lemma}
\begin{proof}
It is clear that $U'$ is $G$-invariant. Since $U' \neq \emptyset$, it contains a fixed point $\alpha$
and by Lemma \ref{GcatFix} there exits a $G$-homotopy $H:U \times I \to X$, such that for all $x\in U$ we have $H(x,0) = x$ and $H(x,1) = \alpha$. Take the restriction of $H$ to $U'$
$$ H \bigg|_{U'} =H' : U' \times I \longrightarrow X^G , \quad\quad H'(x,t) = H(x,t).$$
$H'$ is well-defined because for every $x\in U' = U \cap X^G$, we have
$$ g.H'(x,t) = g.H(x,t) = H(g.x,t) = H(x,t) = H'(x,t) $$
for all $g \in G $ and $  t\in I$.
Therefore the inclusion of $U'$ in $X^G$ is $G$-contractible to $\oO(\alpha) = \{\alpha\}$.
\end{proof}
%==============
\begin{corollary}\label{lbcat}
Suppose $\{U_i\}_{i=1}^n$ is a $G$-categorical covering of $X$. 
Then ${\{U_i \cap X^G\}_{i=1}^n}$ is a $G$-categorical covering of $X^G$ and therefore
$$ \Big| \pi_0(X^G) \Big| \leq  cat(X^G) = cat_G(X^G) \leq cat_G(X).$$
\end{corollary}
Note that the previous corollary also follows from \cite{HT}.
%==============
\begin{lemma}
If $|X^G|<\infty$, then every $G$-categorical set contains at most one fixed point. So all fixed points are isolated fixed points and we have 
$|X^G|=cat_G(X^G) = cat(X^G) $.
\end{lemma}
\begin{proof}
Since $X$ is Hausdorff and $|X^G|< \infty $, every $x\in X^G$ is an isolated fixed point. Thus the statement
follows from Lemma \ref{isolatpt}.

\end{proof}
%==============

%==============
\begin{lemma}
Let $\alpha$ and $\beta$ be two distinct fixed points belong to a path-component of $X^G$. If $U$ and $W$ are two disjoint subsets of $X$ which are $G$-contractible to $\alpha$ and $\beta$ respectively, then $U \cup W$ is $G$-contractible to $\alpha$.
\end{lemma}
\qed
%%==============
%==============
\begin{defn}
Let $G$ be a topological group acting on a topological space $X$. The sequence 
$$ \emptyset=A_0 \varsubsetneq A_1 \varsubsetneq A_2 \varsubsetneq \cdots \varsubsetneq A_n = X $$ 
of open sets in $X$ is called $G$-categorical sequence or simply $G$-cat sequence of length $n$ if
\begin{itemize}
\item each $A_i$ is $G$-invariant, and
\item for each $1 \leq i \leq n$, there exists a $G$-categorical subset $U_i$ of $X$, such that
$$ A_i - A_{i-1} \subset U_i .$$
\end{itemize}
A $G$-cat sequence of length $n$ is called minimal if there exists no $G$-cat sequence with smaller length in $X$.
\end{defn}
%==============
\begin{lemma}\label{G-seq}
Let $G$ be a topological group acting on a topological space $X$. Then there exists a minimal $G$-cat sequence of length $n$ in $X$, if and only if
$$ cat_G(X) = n.$$
\end{lemma}
\begin{proof}
This is analogous to the proof for classical cat \cite[Lemma 1.36]{CLOT}. 
\end{proof}
%%=============
\begin{defn}
 A $G$-$path$ from an orbit $\oO(x)$ to an orbit $\oO(y)$ is a $G$-homotopy 
$H: \oO(x) \times I \to X$ such that the following hold:
\begin{enumerate}
\item $H_0$ is the inclusion of $\oO(x)$ in $X$.
\item $H_1(\oO(x)) \subseteq \oO(y)$.
\end{enumerate}
\end{defn}
%==============
\begin{lemma}(Lemma 3.2, \cite{HT})\label{ht}
Let $H: \oO(x) \times I \to X$ be a $G$-path in $X$ and $x_t= H(x, t)$. Then $G_x \subseteq G_{x_t}$
for all $0 \leq t \leq 1$. 
%% Hence $\dim{\oO(x_t)} \leq \dim{\oO(x)}$.
\end{lemma}
%==============
\begin{lemma}\label{ghp}
 Let $\oO(x)$ and $\oO(y)$ be two distinct orbits in a $G$-space $X$. If $\oO(x)$ and $\oO(y)$ both sit inside a
 $G$-categorical subset, then there exist an orbit $\oO(z)$ such that there are $G$-paths from $\oO(x)$ to
 $\oO(z)$ and $\oO(y)$ to $\oO(z)$.
 
\end{lemma}
\begin{proof}
It is clear from the definition of $G$-categorical open subset. 
\end{proof}
%==============
%==============
\begin{defn}
A $G$-space $X$ is called $G$-connected if for every closed subgroup $H$ of $G$, $X^H$ is path-connected.
\end{defn}
%==============
\begin{lemma}\label{orbit}
(\cite{CG} Lemma 3.14) Let $X$ be $G$-connected, and let $x,y \in X$ such that $G_x \subset G_y$. Then there exists a $G$-path from $\oO(x)$ to $\oO(y)$.
\end{lemma}
%==============
\begin{lemma}
Let $X$ and $Y$ be $G$-connected. Then $X \times Y$ with diagonal action is $G$-connected.
\end{lemma}
\begin{proof}
If $H$ is a closed subgroup of $G$, then $(X \times Y)^H = X^H \times Y^H$.
\end{proof}
%==============
\begin{lemma}\label{G-contract}
Let $X$ be a $G$-connected space with $\alpha \in X^G \neq \emptyset$. Then every $G$-categorical subset $U$ in $X$ is equivariantly contractible to $\alpha$.
\end{lemma}
\begin{proof}
Let $F:U\times I \to X$ be a $G$-homotopy such that $F(x,0) = x$ and $F(x,1) \in \oO(z)$, for some $z \in X$. Since $G_z$ is a subset of $G_{\alpha}=G$, and $X$ is $G$-connected, by Lemma \ref{orbit}, there exists a $G$-homotopy $E:\oO(z) \times I \to X$ so that $E(y,0)=y$ and $E(y,1) = \alpha$. Define the desired $G$-homotopy $H:U\times I \to X$ by

$$H(x,t) = \left\{
\begin{array}{ll}
F(x,2t), & 0\leq t\leq \frac{1}{2}\\
E(F(x,1) , 2t-1), &  \frac{1}{2} \leq t \leq 1
\end{array} \right.
$$
and the lemma follows.
\end{proof}
%==============
By using Lemma \ref{G-contract} one can show that if $X$ is a $G$-connected space with $\alpha \in X^G \neq \emptyset$, then for every two disjoint $G$-categorical subset $U$ and $W$ in $X$, $U\cup W$ is equivariantly contractible to $\alpha$. Also for every $G$-categorical subset $V$ in $Y$, where $Y$ is another $G$-connected space with $\beta \in Y^G \neq \emptyset$, $U \times V$ is equivariantly contractible to $(\alpha, \beta)$.\\
%==============
%==============
%==============
\begin{defn}~
\begin{itemize}
\item A topological space $X$ is called completely normal if for every two subsets $A$ and $B$ of $X$ with
$$\overline{A} \cap B = \emptyset \quad, \quad A \cap \overline{B} = \emptyset, $$
there exist two disjoint open subsets containing $A$ and $B$.
\item A $G$-space $X$ is called $G$-completely normal if for every two $G$-invariant subsets $A$ and $B$ of $X$ with
$$\overline{A} \cap B = \emptyset \quad, \quad A \cap \overline{B} = \emptyset $$
there exist two disjoint $G$-invariant subsets containing $A$ and $B$.
\end{itemize}
~\\
Note that every metric space is completely normal.
\end{defn}
%==============
\begin{lemma}(\cite{CG} Lemma 3.12 )\label{GcomNor}
If $X$ is a completely normal $G$-space, then $X$ is $G$-completely normal.
\end{lemma}
%==============

%==============
\begin{theorem}\label{GcatProd}
Let $X$ and $Y$ be $G$-connected such that $X\times Y$ is completely normal. If $X^G \neq \emptyset$ and $Y^G \neq \emptyset$, then 
$$ cat_G(X\times Y) \leq cat_G(X) + cat_G(Y)-1 , $$
where $X\times Y$ is given the diagonal $G$-action.
\end{theorem}
\begin{proof}
The idea of proof is similar to the proof for classical cat, \cite[Theorem 1.37]{CLOT}.
Let $\alpha \in X^G$, $\beta \in Y^G$, $cat_G(X) =n$, and $cat_G(Y) = m$. So by Lemma \ref{G-seq} there exist $G$-cat sequences of length $n$ and $m$: 
$$
\begin{array}{cc}
\emptyset = A_0 \subset A_1 \subset \cdots \subset A_n = X \; ,\\
\emptyset = B_0 \subset B_1 \subset \cdots \subset B_m = Y.
\end{array}
$$
Denote the $G$-categorical subsets containing the differences by
$$ A_i - A_{i-1} \subset U_i \quad \text{and} \quad B_j - B_{j-1} \subset W_j.$$
Define subsets of $X \times Y$ by
$$ C_0 = \emptyset \; , \quad C_1 = A_1 \times B_1 \; , \quad C_k = \bigcup_{i=1}^k A_i \times B_{k+1-i} \; , \quad C_{n+m-1} = A_n \times B_m = X \times Y, $$
where $A_i = \emptyset$ if $i>n$, and $B_j = \emptyset$ if $j>m$. Note that $C_k$ is $G$-invariant and
$$ C_k - C_{k-1} = \bigcup_{t=1}^{k} (A_t - A_{t-1}) \times (B_{k+1-t} - B_{k-t}).$$
Also for any $k$ and $t$,
$$ (A_t - A_{t-1}) \times (B_{k+1-t} - B_{k-t}) \subset U_t \times W_{k+1-t} ,$$
where $U_t \times W_{k+1-t}$ is a $G$-categorical subset of $X \times Y$ contracting to $(\alpha,\beta)$. Although for a fixed $k$ and varying $t$ there may be intersections among these sets, but this issue can be resolved by using the assumption that $X\times Y$ is $G$-completely normal. Denote 
$$ \Sigma_i = (A_i - A_{i-1}) \times (B_{k+1-i} - B_{k-i}).$$
Since for $i \neq j$ we have 
$$ \overline{\Sigma_i} \; \cap \; \Sigma_j = \emptyset \quad \text{and} \quad \Sigma_i \; \cap \; \overline{\Sigma_j} = \emptyset ,$$
and $X \times Y$ is $G$-completely normal, there exist disjoint $G$-invariant neighborhoods about  $\Sigma_i$ and $\Sigma_j$. By taking the intersection of those disjoint neighborhoods with $U_i \times W_{k+1-i}$ and  $U_j \times W_{k+1-j}$, we obtain disjoint $G$-categorical neighborhoods of $\Sigma_i$ and $\Sigma_j$, for $i\neq j$. So each $C_k - C_{k-1}$ sits inside a $G$-categorical subset of $X \times Y$, and therefore 
$$ \emptyset = C_0 \subset C_1 \subset \cdots \subset C_{m+n-1} = X \times Y $$
is a $G$-cat sequence for $X \times Y$. Thus
$$ cat_G(X \times Y) \leq n+m-1. $$
\end{proof}
%==============
We remark that in \cite{CG} the authors have a similar statement (Theorem 3.15), however there the assumption on fixed point set is not enough and leads to counterexamples (See Example \ref{e4}).

%**********************************************************************************
%***************        2- Locally Standard Torus Manifolds       *****************
%**********************************************************************************
\section{Locally Standard Torus Manifolds}\label{sec2}

%==============
Following \cite{Da} we recall the definition of nice manifold with corners.
An $n$-dimensional manifold with corners is a Hausdorff second-countable topological space together with a maximal atlas of local charts onto open subsets of $\RR^n_{\geq 0}$ such that the overlap maps are homeomorphisms which preserve codimension function. Codimension function $c$ at the point 
$$ x=(x_1, \cdots , x_n) \in \RR^n_{\geq 0} ,$$
is the number of $x_i$ which are zero.
That means the codimension function is a well defined map from manifold with corners $P$ to non-negative integers. 
A connected component of $c^{-1}(m)$ is called an open face of $P$. The closure of an open face is called
a face. Note that we can talk about the dimension of faces of $P$. For example the dimension of $c^{-1}(m)$ is $n-m$.
A $0$-dimensional face is called a vertex and a codimension one face is called a facet of $P$. 

The manifold with corners $P$ is called $nice$ if for every $p \in P$ with $c(p)=2$, the number of codimension one face of $P$ which contains $p$ is also 2.
Therefore a codimension-$k$ face of the nice manifold with corners $P$ is a connected component of the intersection of unique collection
of $k$ many codimension one faces of $P$. An example of manifold with corner which is not nice can be found in Section 6 of \cite{Da}.
The boundary of an $n$-dimensional manifold with corners is the correspondent set of points in local charts for which the codimension function is at least one.

An $n$-dimensional $simple$ polytope in $\RR^n$ is a convex polytope where exactly $n$ bounding hyperplanes meet at 
each vertex. It is easy to see that simple polytope is a nice manifold with corners. For notational purposes, we consider 
a nice manifold with corners as a polytope if it is homeomorphic to a simple polytope and the codimension function is preserved.
We denote the set of vertices of a nice manifold with corners $P$ by $V(P)$ and the set of facets of $P$ by $\mathcal{F}(P)$.

%==============
\begin{defn}
A smooth action of $\TT^n$ on a $2n$-dimensional smooth manifold $M$ is said to be locally
standard if every point $y \in M $ has a $\TT^n$-invariant open neighborhood
$U_y$ and a diffeomorphism $\psi : U_y \to V$, where $V$ is a $\TT^n$-invariant
open subset of $\CC^n$, and an isomorphism $\delta_y : \TT^n \to \TT^n$ such that
\ $\psi (t\cdot x) = \delta_y (t) \cdot \psi(x)$ for all $(t,x) \in \TT^n \times U_y$. 
\end{defn}
%==============
Modifying the definition of quasitoric manifold and torus manifold in \cite{BP} and \cite{HM}, we
consider the following. More general torus actions are discussed in \cite{Yo} by Yoshida.
%==============
\begin{defn}\label{qtd02}
A closed, connected, oriented, and smooth $2n$-dimensional $\TT^n$-manifold $M$ is called a locally standard torus manifold over a nice manifold with corners $P$ if
the following conditions are satisfied:
\begin{enumerate}
\item The $\TT^n$-action is locally standard.
\item $ \partial P \neq \emptyset$, where $\partial{P}$ is the boundary of $P$.
\item There is a projection map $\mathfrak{q}: M \to P$ constant on orbits which maps every
$l$-dimensional orbit to a point in the interior of an $l$-dimensional face of $P$.
\end{enumerate}
In the case that $P$ is a polytope, $M$ is called a quasitoric manifold.
\end{defn}
%==============

Note that according to the Definition \ref{qtd02}, $P$ is the orbit space and is path-connected. 
Also we remark that for the definition of torus manifolds in \cite{HM}, the authors assume that the torus action has fixed points.
But here we do not have such restrictions.
%==============
%==============
\begin{example}\label{33}
Consider the natural $\TT^n$-action on
$$ \zS^{2n}=\{(z_1, \ldots, z_n, x) \in \CC^n \times \RR : |z_1|^2+ \cdots + |z_n|^2+x^2 =1\},$$
which is defined by 
$$(t_1, \ldots, t_n) \cdot (z_1, \ldots, z_n, x) \mapsto (t_1z_1, \ldots, t_nz_n, x).$$ 
The orbit space is given by $Q=\{(x_1, \ldots, x_n, x) \in \zS^n : x_1, \ldots, x_n \geq 0\}$ and
the number of fixed points is $2$.
%%%---------------------------

This action is a locally standard action, so $S^{2n}$ is a locally standard torus manifold.
 Note that $S^{2n}$ is not a quasitoric manifold if $n \geq 2$.
When $n=2$ the orbit space is an eye shape. 
\end{example}

\begin{example}\label{consum}
Let $M_1$ and $M_2$ be two quasitoric manifolds of dimension $2n$, and $\TT^k$ be the $k$-dimensional torus,
$0\leq k \leq n$.
Let $\phi_i : \TT^k \to M_i$ be the embedding onto $k$-dimensional orbit of $M_i$, and let $\tau_i$ be the invariant tubular neighborhood of $\phi_i(\TT^k)$ for $i=1, 2$. Identifying the boundary of $\tau_1$ in $M_1$ and $\tau_2 $ in $M_2$ via an equivariant diffeomorphism, we get a smooth $\TT^n$-manifold, which is called an equivariant connected sum of $M_1$ and $M_2$, denoted $M_1 \#_{\TT^k} M_2$.
Clearly $M_1 \#_{\TT^k} M_2$ is a torus manifold, and it is not a quasitoric manifold if $k \geq 1$. Note that the above construction depends on the isomorphism type of the isotropy representations and on the gluing map. Here we are assuming that the isotropy representations are the same and the gluing map is the natural one.

A more general equivariant connected sum of smooth manifolds with torus action is described in \cite{GK}.
Equivariant connected sum of quasitoric manifolds at a fixed point and along a principal orbit is
discussed in \cite{BR} and \cite{PS} respectively.
\end{example}

%==============
\begin{defn} \label{charfun}
A function $\lambda : \mathcal{F}(P) \to \ZZ^n$ is called characteristic function if the submodule generated by
$\{\lambda(F_{j_1}), \ldots, \lambda(F_{j_l})\}$ is an $l$-dimensional direct summand of $\ZZ^n$ whenever
the intersection of the facets $F_{j_1}, \ldots, F_{j_l}$ is nonempty.

The vectors $\lambda_j = \lambda(F_{j})$ are called characteristic vectors and the pair $(P, \lambda)$ is called a characteristic pair.
\end{defn}
%==============
In \cite{MP} the authors show that given a torus manifold with locally standard action one can associate a
characteristic pair to it up to the choice of sign of characteristic vectors. They also constructed a
torus manifold with locally standard action from the pair $(P, \lambda)$. Following \cite{MP} we write
the construction briefly. A more general construction is done in \cite{Yo}.

Let $P$ be a nice manifold with corners and $(P, \lambda)$ be a characteristic pair.
A codimension-$k$ face $F$ of $P$ is a connected component of the intersection $F_{j_1} \cap \ldots \cap F_{j_k}$
of unique collection of $k$ facets $F_{j_1}, \ldots, F_{j_k}$ of $P$.
Let $\ZZ(F)$ be the submodule of $\ZZ^n$ generated by the characteristic vectors $\lambda_{j_1}, \ldots, \lambda_{j_k}$.
Then $\ZZ(F)$ is a direct summand of $\ZZ^n$. Therefore the torus $\TT_F : = (\ZZ(F) \otimes_{\ZZ} \RR)/\ZZ(F)$
is a direct summand of $\TT^n$. Define $\ZZ(P)=(0)$ and $\TT_P$ to be the proper trivial subgroup of $\TT^n$.
If $p \in P$, then $p$ belongs to the relative interior of a unique face $F(p)$ of $P$.

Define an equivalence relation $\sim$ on the product $\TT^n \times P$ by
\begin{equation}
(t,p) \sim (s,q) \quad \Longleftrightarrow \quad ~p=q~\mbox{and} ~ s^{-1}t \in \TT_{F(p)}.  \label{equ001} %%\tag{$\ast$}
\end{equation}
Let $$M(P, \lambda) = (\TT^n \times P) / \sim $$ be the quotient space.
The group operation in $\TT^n$ induces a natural $\TT^n$-action on $M(P, \lambda)$.
The projection onto the second factor of $\TT^n \times P$ descends to the quotient map
\begin{equation}
\mathfrak{q} : M(P, \lambda) \to P ,\quad \mathfrak{q}([t,p]) = p \label{equ002}  %%\tag{$\ast\ast$}
\end{equation}
where $[t, p]$ is the equivalence class of $(t, p)$. So the orbit space of this action is $P$.
One can show that the space $M(P, \lambda)$ has the structure of a locally standard torus manifold.
%==============
\begin{defn}
 Two $\TT^n$-actions on $2n$-dimensional torus manifolds $M_1$ and $M_2$ are called equivalent
if there is a homeomorphism $ f : M_1 \to M_2$ such that 
$$ f(t \cdot x) = t \cdot f(x), \quad \forall \; (t, x ) \in \TT^n \times M_1.$$
\end{defn}
%==============
\begin{defn}\label{defdel}
Let $\delta : \TT^n \to \TT^n$ be an automorphism. Two torus manifolds $M_1$ and $M_2$ over the same
manifold with corners $P$ are called $\delta$-equivariantly homeomorphic if there is a homeomorphism $ f : M_1 \to M_2$
such that 
$$f(t \cdot x) = \delta(t)\cdot f(x) , \quad \forall \; (t, x ) \in \TT^n \times M_1.$$

When $\delta$ is the identity automorphism, $f$ is called an equivariant homeomorphism.
\end{defn}

%==============

\begin{prop}\label{clema2}
Let $M$ be a $2n$-dimensional locally standard torus manifold over $P$, and
$\lambda:\mathcal{F}(P) \to \ZZ^n$ be its associated characteristic function. Let
$M(P, \lambda)$ be the locally standard torus manifold constructed
from the pair $(P, \lambda)$, and $H^2(P, \ZZ)=0$. Then there is an equivariant homeomorphism
$f: M(P, \lambda) \to M$ covering the identity on $P$.
\end{prop}
This proposition is a particular case of Theorem 6.2 in \cite{Yo}.  
We remark that this result is proved for quasitoric manifolds in \cite{DJ}, for torus manifolds
with locally standard action in \cite{MP}, and for specific $4$-dimensional manifolds with effective
$\TT^2$-action in \cite{OR}.

%==============
\begin{lemma}\label{simcq}
Let $M_1$ and $M_2$ be $2n$-dimensional quasitoric manifolds, then $M_1 \#_{\TT^k} M_2$ is simply connected for all $n$ and $k$ except $k=n= 2$.
\end{lemma}
\begin{proof}
We adhere the notations of Example \ref{consum}. When $k=n=1$, then $M_1 \#_{\TT^k} M_2 = \zS^2 \sqcup \zS^2$.
So it is simply connected. Now consider the other cases of $n$ and $k$ except $n=k=1, 2$.
 Let $\mathfrak{q}_i: M_i \to P_i$ be the orbit map, and 
$ Q_i = P_i - \mathfrak{q}_i(\tau_i) \simeq P_i -\{\ast\}$ where $\ast \in P_i$ for $i=1,2 $.
Then $Q_i$ is simply connected and $M_i - \tau_i = \mathfrak{q}_i^{-1}(Q_i) $. By Proposition \ref{clema2} we have
$$  M_i - \tau_i \cong \big( \TT^n \times Q_i \big)/_{\sim} $$
where $\sim$ is defined in \eqref{equ001}.

Let $g_i : \TT^n \times Q_i \to M_i - \tau_i $ be the quotient map, for $i=1,2$. By definition of the equivalence relation $\sim$, $g_i^{-1}(x)$ is connected for all $x \in M_i - \tau_i$. Also $ \TT^n \times Q_i $ is locally path-connected and $M_i - \tau_i$ is semi-locally simply connected.
Thus by Theorem 1.1 in \cite{CGM}, we get a surjective map $$\pi_1(g_i) : \pi_1(\TT^n \times Q_i) \twoheadrightarrow \pi_1(M_i -\tau_i).$$ Since $Q_i$ is simply connected, 
$$ \pi_1 (\TT^n \times Q_i) = \pi_1(\TT^n).$$
Existence of fixed point in $M_i - \tau_i$ implies that all generator of $\pi_1(\TT^n)$ maps to zero. So $\pi_1(M_i - \tau_i)$ is trivial. Hence $\pi_1(M_1 \#_{\TT^k} M_2)$ is trivial by Van-Kampen theorem.

\end{proof}
%==============
More generally we have,
%==============
\begin{theorem}\label{bs}
Let $M$ be a locally standard torus manifold with orbit space $P$. If $M$ has a fixed point
and $P$ is simply connected, then $M$ is simply connected.
\end{theorem}
\begin{proof}
Since $M$ is a smooth locally standard torus manifold with fixed point, the orbit space $P$
is a nice manifold with corners and $\partial P \neq \emptyset$ (see Section 4 in \cite{Yo}).

By result of Yoshida \cite{Yo}, $M$ is equivariantly homeomorphic to $T_P /\sim_l$,
where $T_P$ is a principal $\TT^n$-bundle over $P$ and $\sim_l$ is defined in Definition 4.9 in \cite{Yo}.
Since $P$ is simply connected, the fibration 
$$\TT^n \to T_P \to P $$
induces a surjective map $i_{\ast}: \pi_1(\TT^n) \to \pi_1(T_P)$. Let $f: T_P \to T_P/\sim_l$
be the quotient map. From Section 4 of \cite{Yo}, the fiber $f^{-1}(x)$ of each point $x \in T_P/\sim_l$
 is a connected subset of $\TT^n$. Hence by Theorem 1.1 in \cite{CGM},
 $$f_{\ast} : \pi_1(T_P) \to \pi_1(T_P/\sim_l)=\pi_1(M)$$ is surjective and therefore $f_{\ast} \circ i_{\ast}$
 is surjective. Since $\TT^n$-action has a fixed point, all generators of $\pi_1(\TT^n)$ maps to identity via $f_{\ast} \circ i_{\ast}$. Thus $\pi_1(M)$ is trivial.
\end{proof}
%==============

%**********************************************************************************
%*******     4- LS-category of Locally Standard Torus Manifolds         ***********
%**********************************************************************************
\section{LS-category of Locally Standard Torus Manifolds}\label{sec4}

%==============
The Lusternik-Schnirelmann category of a space $X$, denoted $cat(X)$, is the least integer $n$ such that there exists an open covering $U_1, \ldots , U_n$ of $X$ with each $U_i$ contractible to a point in the space $X$. If no such integer exists, we write $cat(X) = \infty$. 

In this section we discuss the LS-category of locally standard torus manifolds for the following cases:
\begin{itemize}
\item Quasitoric manifolds.
%\item Locally standard torus manifold over $P$, where $\partial P$ contains a boundary of a simple polytope.
\item Locally standard torus manifold over $P$, where $P$ is simply connected and a connected component of $\partial P$ is a simple polytope.
\item $4$-dimensional locally standard torus manifold over $P$, where a connected component of $\partial P$ is a boundary of a polygon.
\end{itemize}
%==============
\begin{lemma}
 Let $M$ be a $2n$-dimensional quasitoric manifold over a simple polytope $P$. Then $cat(M) = n+1$.
\end{lemma}
\begin{proof}
By Proposition 3.10 in \cite{DJ}, each generator of degree $2n$ in the integral cohomology group of $M$ is a product of $n$ cohomology classes of lowest dimension 2. Since $dim(M) = 2n$, cuplength of $M$ (see Definition 1.4 of \cite{CLOT}) is $n$,
$$ cup_{\ZZ} (M) = n.$$
Thus by Proposition 1.5 in \cite{CLOT},
$$cat(M) \geq n+1.$$
By Corollary 3.9 of \cite{DJ}, $M$ is simply connected. Therefore by Proposition 27.5 in \cite{FHT},
$$cat(M) \leq n+1.$$
\end{proof}
%==============
%==============
\begin{lemma}\label{lbc}
Let $M$ be a $2n$-dimensional locally standard torus manifold over $P$. If a connected component of $\partial P$ is a boundary of an $n$-dimensional simple polytope $Q$, then 
$$ cat(M) \geq n+1.$$
\end{lemma}

\begin{proof}

Let $v$ be a vertex of $Q$ and $v= F_{i_1} \cap \cdots \cap F_{i_n}$, where $F_{i_1}, \cdots , F_{i_n}$ are unique $n$-many facets of $Q$ (and therefore facets of $P$). Let $x_v = \mathfrak{q}^{-1} (v)$ and $X_j = \mathfrak{q}^{-1} (F_{i_j})$, for $j=1,2,\cdots,n$. Since $\TT^n$-action on $M$ is locally standard, $x_v$ is a fixed point and the intersection $X_1 \cap \cdots \cap X_n (=x_v)$ is transversal. Therefore the Poincar\'{e} dual of $X_j$ represents a non-zero cohomology class in $H^2(X, \ZZ)$ (see Section 0.4 in \cite{GH}). So by definition of cup-length, $ cup_{\ZZ} (M) \geq n$. 
\end{proof}
%==============
~\\
Note that Lemma \ref{lbc} is not true for every locally standard torus manifold,
see the  Example \ref{e3}.

%==============

%==============
\begin{theorem}\label{lsm}
Let $M$ be a $2n$-dimensional locally standard torus manifold with a simply connected orbit space $P$.
If a connected component of $\partial P$ is the boundary of a simple polytope $Q$, then $$ cat(M) = n+1 .$$
\end{theorem}
\begin{proof}
By Theorem \ref{bs} $M$ is simply connected, so $cat(M) \leq n+1$. On the other hand by Lemma \ref{lbc},
$cat(M) \geq n+1$.
 
\end{proof}
%==============
\begin{corollary}
Let $M_1$ and $M_2$ be quasitoric manifolds. Then for any $k$ and $n$ except $k=n=2$, we have
$$ cat(M_1 \#_{\TT^k} M_2 ) = n+1 .$$
\end{corollary}
\begin{proof}
Let $P$ be the orbit space of locally standard $\TT^n$-action on $ M_1 \#_{\TT^k} M_2 $.
Since $M_1$ and $M_2$ are quasitoric manifolds, $\partial P$ contains the boundary of a
simple polytope. Also by Lemma \ref{simcq}, $ M_1 \#_{\TT^k} M_2 $ is simply connected. Therefore by Theorem \ref{lsm}
$$ cat(M_1 \#_{\TT^k} M_2 ) = n+1 .$$
\end{proof}
%==============
\begin{lemma}\label{oc}
Let $M$ be a $4$-dimensional locally standard torus manifold with a fixed point $x_0$. Then any orbit is contractible to $x_0$.
\end{lemma}
\begin{proof}
Let $P$ be the orbit space and $\mf{q} : M \to P$ be the orbit map. By Proposition \ref{clema2}, we may assume that $M = M(P, \lambda)$ where $\lambda$ is the characteristic function of $M$. Let $\theta$ be an orbit such that $\mf{q}(\theta)=x \in P$. We can choose a path $\alpha: [0,1] \to P$ from $x$ to $x_0$ such that $\alpha$ is injective and $\alpha(0,1) \cap P \subset P^0$ (interior of $P$). We denote the image of $\alpha$ by $[x, x_0]$. Then $$(\TT^2 \times [x, x_0])/\sim ~ \subset M. $$ Let $\TT^2_x$ be the isotropy group of $x$. Then $$\theta =\mf{q}^{-1}(x) = (\TT^2 \times x)/\sim \cong \TT^2/\TT^2_x.$$ Since the $\TT^2$-action is locally standard, we have $\TT^2 \cong \TT^2_x \oplus (\TT^2/\TT^2_x)$. Observe that $(\TT^2/\TT^2_x \times [x, x_0])/\sim$ gives a homotopy.   
\end{proof}

%================
\begin{theorem}\label{catM3}
Let $M$ be a 4-dimensional locally standard torus manifold over $P$, such that a connected component of $\partial P$ is the boundary of a polygon. Then 
$$ cat(M) = 3.$$ 
\end{theorem}
\begin{proof}
By Lemma \ref{lbc}, $cat(M) \geq 3$. Since the $\TT^2$-action on $M$ is locally standard, $P$ is a nice
2-dimensional manifold with corners. So every component of $\partial P$ is either boundary of a polygon, a circle, or an eye shape (see Figure \ref{eye}).
%%%---------------------------
\begin{figure}[h]
\centering
\includegraphics[scale=0.6]{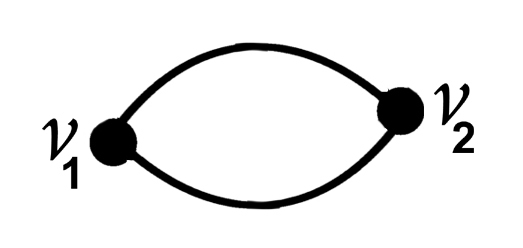}			%%%=====  Image
\caption{An eye shape}
\label{eye}
\end{figure}
%---------------------------

Note that $P$ can be obtained from a closed surface by removing the interior points 
of a finite number of non-intersecting polygons, or polygons and eye shapes, or polygons
and circles, or polygons and eye shapes and circles. Thus by \cite{BCN}
$P$ has a triangulation $\Sigma$
such that the vertices of $P$ belong to the vertex set of $\Sigma$. Let 
\begin{itemize}
\item $\{x_1, \ldots, x_l\}$ be the vertices of $\Sigma$,
\item $\{E_1, \ldots, E_m\}$ be the edges of $\Sigma$, and 
\item $\{F_1, \ldots, F_n\}$ be the faces of $\Sigma$.
\end{itemize}
Suppose $y_j$ and $z_k$ are interior points of $E_j$ and $F_k$ respectively, for $j=1, \ldots, m$ and $k=1, \ldots, n$.
%---------------------------
\begin{figure}[h]
\centering
\includegraphics[scale = 0.3]{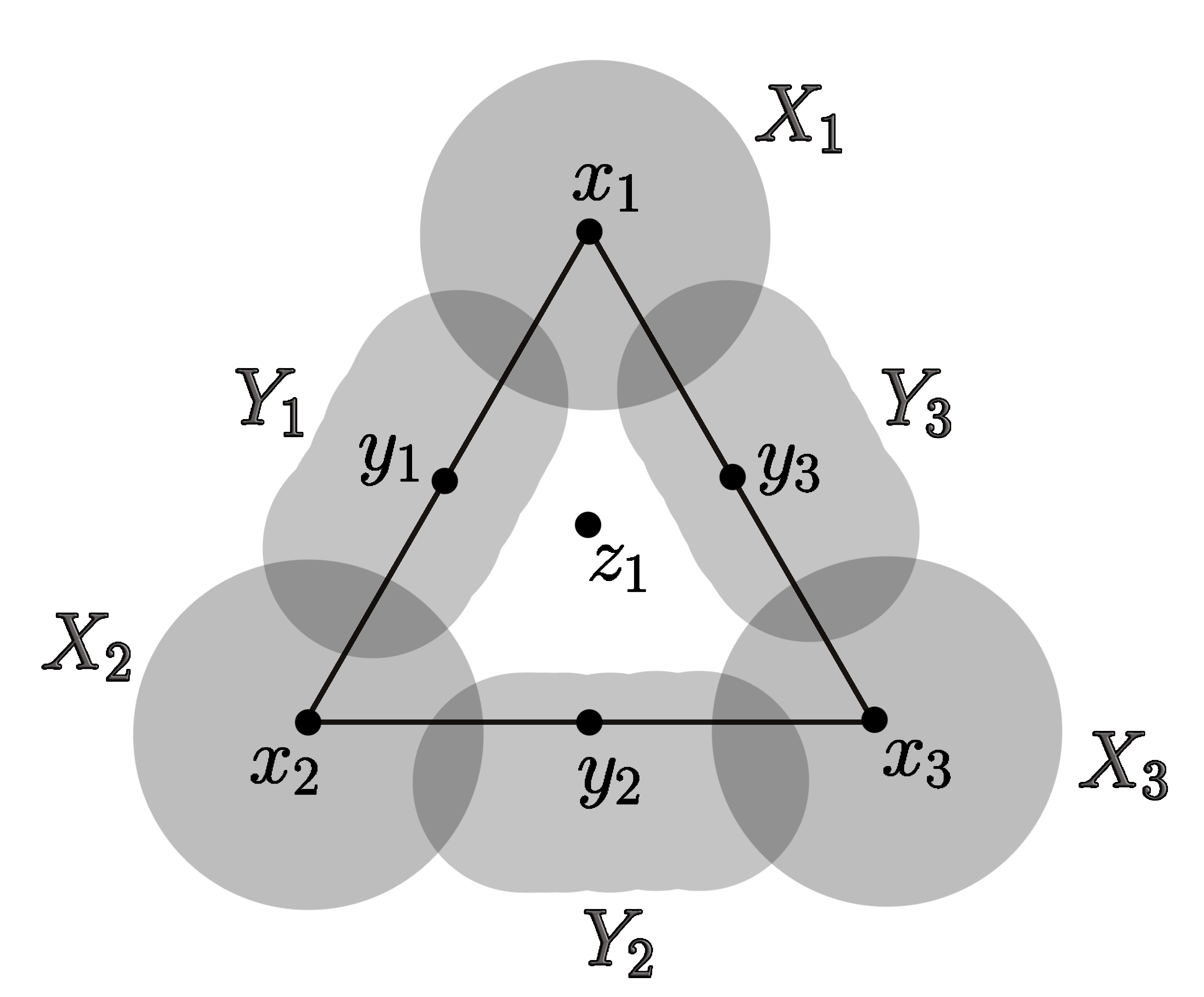}		%%%=====  Image
\caption{Choosing neighborhood $X_i$, $Y_j$, and $Z_k$}
\label{fig:trian}
\end{figure}
%---------------------------
Regarding to the Figure \ref{fig:trian} one can choose the neighborhoods $X_i, Y_j, Z_k$ of $x_i, y_j, z_k$ in $P$ respectively such that 
\begin{enumerate}
\item $X_i \cap X_j = \emptyset$, $Y_i \cap Y_j = \emptyset$ and $Z_i \cap Z_j = \emptyset$ if $i \neq j$.
\item $y_i, z_i \notin X_j$, $x_i, z_i \notin Y_j$ and $x_i, y_i \notin Z_j$ for all $i , j$.
\item $X_{i_1} \cup Y_j \cup X_{i_2}$ is an open neighborhood of $E_j$ in $P$ if $x_{i_1}$ and $x_{i_2}$ are vertices of $E_j$.
\item $Z_k \cup Y_{k_1} \cup Y_{k_2} \cup Y_{k_3} \cup X_{i_1} \cup X_{i_2} \cup X_{i_3}$ is an open neighborhood of $F_k$ in $P$ if $E_{k_1}, E_{k_2}, E_{k_3}$ are edges of $F_k$ and $x_{i_1}, x_{i_2}, x_{i_3}$ are vertices of $F_k$.
\item $Z_k \subset F_k^0$ where $F_k^0$ is the interior of $F_k$. 
\item Each $X_i$ is either homeomorphic (preserving the codimension function) to $\RR^2_{\geq 0}$, or $\RR_{\geq 0} \times \RR$, or $\RR^2$. 
\item Each $Y_j$ is either homeomorphic (preserving the codimension function) to $\RR_{\geq 0} \times \RR$, or $\RR^2$. 
\item Each $Z_k$ is homeomorphic (preserving the codimension function) to $\RR^2$. 
\end{enumerate} 
(See Figure \ref{fig:triangulation}).
%---------------------------
\begin{figure}[h]
\centering
\includegraphics[scale = 0.3]{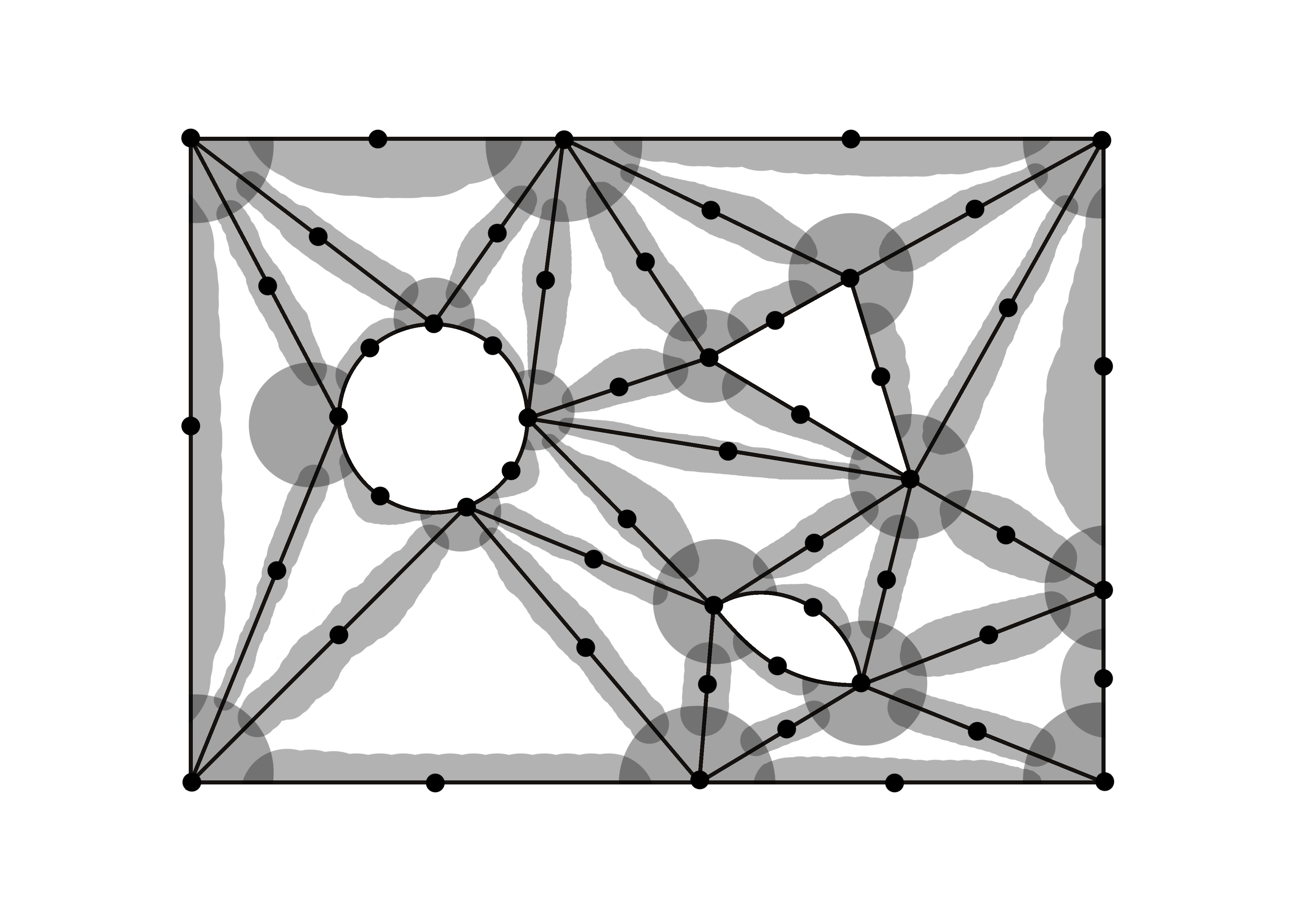}		%%%=====  Image
\caption{Example of covering for a triangulation.}
\label{fig:triangulation}
\end{figure}
%---------------------------

Suppose $\mf{q} : M \to P$ is the orbit map. Let $U_i=\mf{q}^{-1}(X_i)$, $V_j=\mf{q}^{-1}(Y_j)$ and $W_k=\mf{q}^{-1}(Z_k)$ 
for $i=1, \ldots, l$, $j=1, \ldots, m$ and $k=1, \ldots, n$. Then $U_i, V_j$ and $W_k$ are equivariantly contractible to the orbit $\mf{q}^{-1}(x_i), \mf{q}^{-1}(y_j)$, and $\mf{q}^{-1}(z_k)$ respectively. By hypothesis $M$ has a fixed point say $\hat{x}_0$. By Lemma \ref{oc} $\mf{q}^{-1}(x_i), \mf{q}^{-1}(y_j)$, and $\mf{q}^{-1}(z_k)$ are contractible to $\hat{x}_0$. Thus $U_i, V_j$ and $W_k$ are equivariantly contractible to $\hat{x}_0$. Let $$A=\bigcup_{i=1}^l U_i, \quad B= \bigcup_{j=1}^m V_j \quad \mbox{and}\quad C=\bigcup_{k=1}^n W_k.$$ By the choice of $X_i$, $Y_j$ and $Z_k$ we get that $A, B$ and $C$ are contractible to $\hat{x}_0$. Clearly $M=A \cup B \cup C$. Therefore $cat(M) \leq 3$.
\end{proof}

%==============
We remark that the proof of previous theorem could be obtained by using Corollary 1.7 of \cite{OW}, however the current version of proof plays an important role in proof of Theorem \ref{qtm4}.
More generally we can prove the following.
%==============
\begin{corollary}\label{ubc}
Let $M$ be a $2n$-dimensional locally standard torus manifold over $P$. If there exists a triangulation for $P$, then 
$ cat(M) \leq n+1.$ 
\end{corollary}

%===============
\begin{corollary}
Let $M$ be a $2n$-dimensional locally standard torus manifold over $P$, such that a connected component of $\partial P$ is the boundary of a polygon. If there exists a triangulation for $P$, then 
$ cat(M) = n+1.$ 
\end{corollary}
\begin{proof}
This follows from Lemma \ref{lbc} and Corollary \ref{ubc}.
\end{proof}
%==============
%\begin{remark}

Note that Theorem \ref{catM3} is not true for every locally standard torus manifold,
see Examples \ref{22} and \ref{e3}.

%\end{remark}
%==============
%==============
\begin{example}\label{22}
Consider the annulus $P$ and characteristic function $\lambda$ as in the Figure \ref{ann}. Note that $P \cong C \times I$ where $C$ is a circle and $I$ is the closed interval $[0,1]$. Then the following is an equivariant homeomorphism 
$$(\TT^2 \times C \times I)/\sim \quad \cong \quad C \times (\TT^2 \times I)/\sim$$ 
where $\sim$ is defined in \eqref{equ001}. By Section 2 in \cite{OR}, 
$$(\TT^2 \times I)/\sim \quad \cong \quad \mathbb{RP}^3 .$$ 
Therefore
$$ M(P, \lambda) \quad \cong \quad (\TT^2 \times C \times I)/\sim \quad \cong \quad C \times (\TT^2 \times I)/\sim \quad \cong \quad \zS^1\times \mathbb{RP}^3 .$$ 
%where each is an equivariant homeomorphism. 
Since $cat(\mathbb{RP}^3) = 4$ and $cat(\zS^1) =2$, using categorical sequences (see Section 1.5 in \cite{CLOT}),
one can show that
$$ cat(\zS^1\times \mathbb{RP}^3 ) \leq 5.$$ 
On the other hand by K\"{u}nneth theorem,
$$ H^{\ast}(\zS^1 \times \mathbb{RP}^3, \ZZ_2) = H^{\ast}(\zS^1 , \ZZ_2) \otimes_{\ZZ_2} H^{\ast}(\mathbb{RP}^3,\ZZ_2)$$
Therefore $ cup_{\ZZ_2}(\zS^1 \times \mathbb{RP}^3) = 4 $. Thus by Proposition 1.5 in \cite{CLOT}, 
$$ cat(\zS^1\times \mathbb{RP}^3 ) = 5 .$$
%%---------------------------
\begin{figure}[h]
\centering
\includegraphics[scale=0.2]{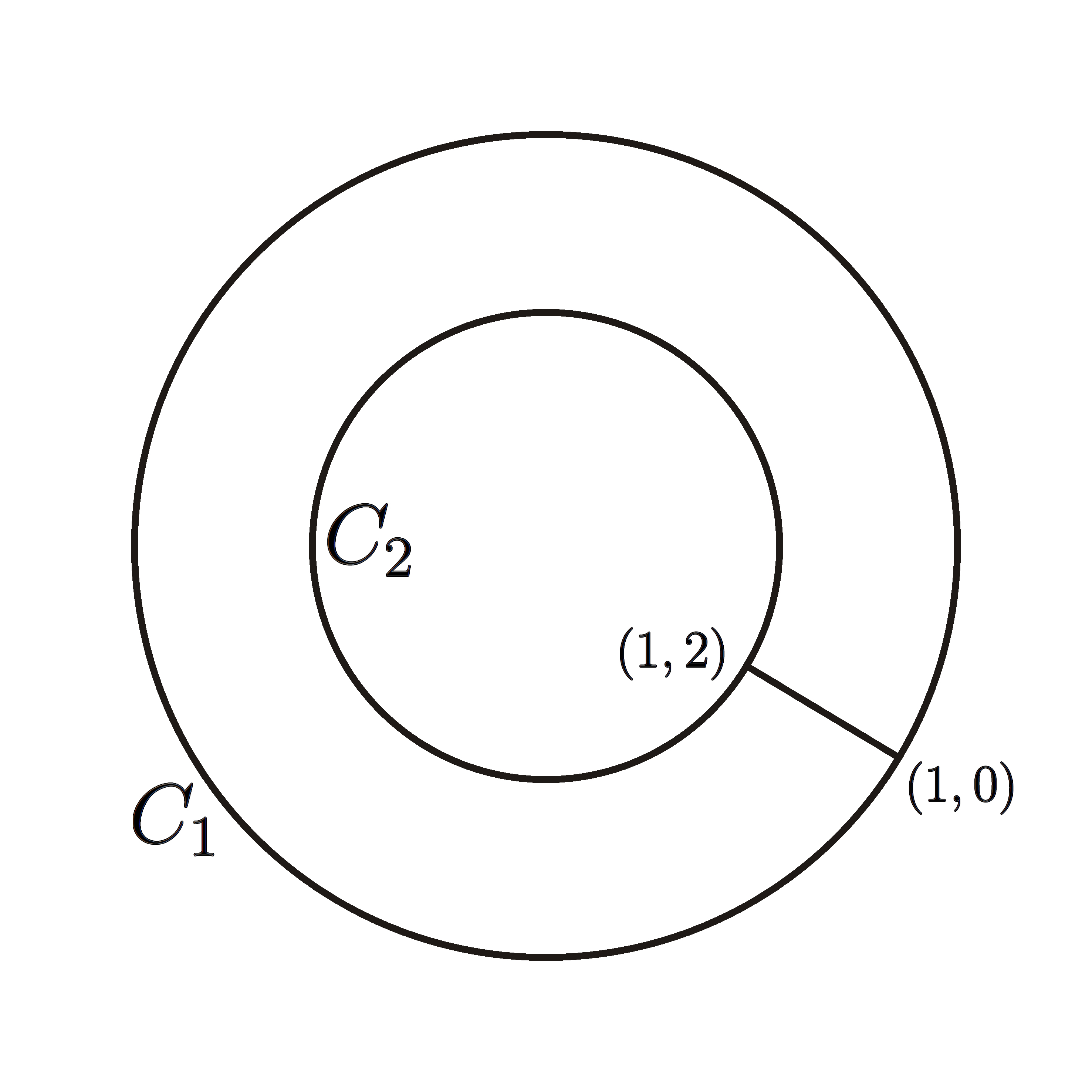}	 %%%=====  Image
\caption{An annulus in $\RR^2$.}
\label{ann}
\end{figure}

\end{example}

%**********************************************************************************
%***    5- Equivariant LS-category of Locally Standard Torus Manifolds     ********
%**********************************************************************************
\section{Equivariant LS-category of Torus Manifolds}\label{sec5}

In this section, we compute equivariant LS-category of some locally standard torus manifolds.

%==============
\begin{theorem}\label{cat-quasi}
Let $M$ be a $2n$-dimensional quasitoric manifold with $k$ fixed points. Then 
$$ cat_{\TT^n} (M) = k.$$
\end{theorem}
\begin{proof}
Since the fixed points are isolated, by Corollary \ref{lbcat} we have 
$$ cat_{\TT^n} (M) \geq k. $$
So it is enough to show that for any $v \in M^{\TT^n}$, there is a $\TT^n$-categorical subset $X_v$, such that
$$ M = \bigcup_{v \in M^{\TT^n}} X_v . $$ 
Let $\mathfrak{q} : M \to P$ be the orbit map. Then $P$ is a simple $n$-polytope and also $M^{\TT^n}$ corresponds bijectively to $V(P)$, the vertex set of $P$. So we may assume 
$$ M^{\TT^n} = V(P) .$$
For $v \in V(P)$, let
$$C_v = \bigcup_{v\not\in F} F , \quad U_v = P - C_v , \quad \text{and}\quad X_v = \mathfrak{q}^{-1} (U_v)$$
where $F$ is a face of $P$. Clearly $X_v$ is $\TT^n$-invariant. Since $U_v$ is a convex subset of $P$, it is contractible to $v$. So there exists a homotopy $h: U_v \times I \to P$ such that for all $x \in U_v$, $h(x,0)=x$, $h(x,1)= v$,
and also for any face $F$ of $U_v$ we have 
$$h(x,t) \in F , \quad \forall x\in F , t\in I .$$
By Lemma 1.8 of \cite{DJ}, 
$$ M \cong M(P,\lambda) \quad \text{and} \quad X_v \cong (\TT^n \times U_v)/\sim $$
where $\lambda$, $M(P,\lambda)$, and $\sim$ are recalled in \eqref{equ001}.
Therefore $h$ induces a homotopy 
$$ Id \times h : \TT^n \times U_v \times I \to \TT^n \times P $$
defined by $ ((t,x) ,r) \mapsto (t, h(x,r))$. Since for each face $F$ of $U_v$, we have 
$$ x \in F \Longrightarrow h(x,r) \in F , \quad\text{for all}\; r\in I,$$
$Id \times h$ induces a homotopy $H: X_v \times I \to M$, with $([t,x],r)\mapsto [t,h(x,r)]$.
Since
$$sH([t,x],r) = s[t,h(x,r)] = [st,h(x,r)] = H([st,x],r) = H(s[t,x],r),$$
therefore $H$ is $\TT^n$- homotopy. Also 
$$H(x,0)=x, \quad H(x,1)=\mathfrak{q}^{-1} (v)=\{v\}, \quad\forall x \in X_v.$$
Thus $X_v$ is $\TT^n$-categorical subset of $M$. Clearly $\{X_v: v\in V(P)\}$
covers $M$, therefore $cat_{\TT^n} (M) = \big| V(P) \big| = k$.
\end{proof}
%==============
\begin{theorem}\label{con1}
Let $M_i $ be a $2n$-dimensional quasitoric manifold over $P_i$, for i=1,2. Then
$$ cat_{\TT^n} (M_1 \#_{\TT^k} M_2) = \big| V(P_1)\big| + \big| V(P_2)\big| \; , \quad\text{for}\; k\geq 1.$$
\end{theorem}

\begin{proof}
We adhere the notations of Example \ref{consum} and Theorem \ref{cat-quasi}.
By the construction
of equivariant connected sum we have $M_1 \#_{\TT^k} M_2$ is a locally standard torus manifold.
 Let  $k \geq 1$.
Then the number of fixed points of $\TT^n$-action on $M_1 \#_{\TT^k} M_2$ is $| V(P_1)| + |V(P_2)|$.
So by Corollary \ref{lbcat}, we have $$cat_{\TT^n}(M_1 \#_{\TT^k} M_2) \geq |V(P_1)| + |V(P_2)|.$$

 Let $\mathfrak{q}_i: M_i \to P_i$
be the orbit map and $\mathfrak{q}_i(\TT^k)=x_i$, so $x_i$ belongs to the relative interior
of a $k$-dimensional face $E_i$ of $P_i$ for $i=1, 2$. Let $\mathcal{L}(P_i)$ be the
face lattice of $P_i$ and $v \in V(P_i)$. Define
$$C_v=\bigcup_{v \notin F \in \mathcal{L}(P_i)} F, \quad U_v = P_i - C_v \quad \mbox{and} \quad X_v= \mathfrak{q}_i^{-1}(U_v).$$

Let $S_1=\{v_{11}, \ldots, v_{1p}\}$ and $S_2=\{v_{21}, \ldots, v_{2q}\}$ be the vertices of $E_1$
and $E_2$ respectively. For $i \in \{1, 2\}$, let $$\alpha_{ij} : I \to P_i$$ be a simple path
from $x_i$ to $v_{ij}$ such that:
\begin{itemize}
\item $\alpha_{ij}(I^0) \subset E_i^0$, where $E_i^0$ is the relative interior of $E_i$, and
\item $\alpha_{i1}(I^0) \cap \alpha_{i2}(I^0) = \emptyset$,
\end{itemize}
where $1 \leq j \leq p$ for $i=1$ and $1 \leq j \leq q$ for $i=2$.
Let
\begin{equation}
V_v = \left\{ \begin{array}{ll} U_v - \mathfrak{q}_{i}(\tau_i) & \mbox{if}~v \in V(P_i)-S_i~ \mbox{for}~ i \in \{1, 2\}\\
U_v - \{\mathfrak{q}_i(\tau_i) \cup \alpha_{il}(I^0)\} & \mbox{if}~v \in S_i~ \mbox{and}~v \neq v_{il}.
\end{array}\right.
\end{equation}
Let $P_1 \# P_2$ be the orbit space $(M_1 \#_{\TT^k} M_2)/\TT^n$. Note that $P_1 \# P_2$
 can be obtained from $P_1 - \mf{q}_1(\tau_1)$ and $P_2 - \mf{q}_2(\tau_2)$ by gluing
 $\mf{q}_1(\partial \tau_1)$ and $\mf{q}_2(\partial \tau_2)$ via a homeomorphism which
 preserve codimension function as well as characteristic function. So $V_v$ is an open
 subset of $P_1 \# P_2$ containing the vertex $v$. If $v \in V(P_i)$, then $Y_v = \mathfrak{q}_i^{-1}(V_v)$ is a $\TT^n$-invariant subset of $M_i$ which is equivariantly contractible to the fixed point $\mathfrak{q}_i^{-1}(v)$ by
 Proof of Theorem \ref{cat-quasi}. From the definition of equivariant connected sum,
there is a $\TT^n$-invariant open neighborhood $\widehat{Y}_v$ of $Y_v$ with a $\TT^n$-homotopy from
 $\widehat{Y}_v$ to $Y_v$. Then the collection 
$$\Big\{\widehat{Y}_v : v \in V(P_1) \cup V(P_2)\Big\}$$ is a $\TT^n$-categorical covering of $M_1 \#_{\TT^k} M_2$. Thus 
$$cat_{\TT^n}(M_1 \#_{\TT^k} M_2) \leq |V(P_1)| + |V(P_2)|.$$
\end{proof}
%==============

%==============
\begin{remark}
If $k=0$, then $M_1 \#_{\TT^k} M_2$ is a quasitoric manifold, therefore we can apply Theorem \ref{cat-quasi}.
\end{remark}
\begin{example}
Let $M_1$ and $M_2$ be $4$-dimensional quasitoric manifolds over triangle $P_1$, and rectangle $P_2$ respectively.
 Let $x_i$ be the interior point of $P_i$, $i=1, 2$. Then $\mathfrak{q}_i(\tau_i)$ is a neighborhood of $x_i$ with
 the boundary $C_i$ for $i=1,2$. Regarding to Theorem \ref{con1} here $E_1 = P_1$ and $E_2 = P_2$. 
%%---------------------------
\begin{figure}[h]
\centering
\includegraphics[scale = 0.6]{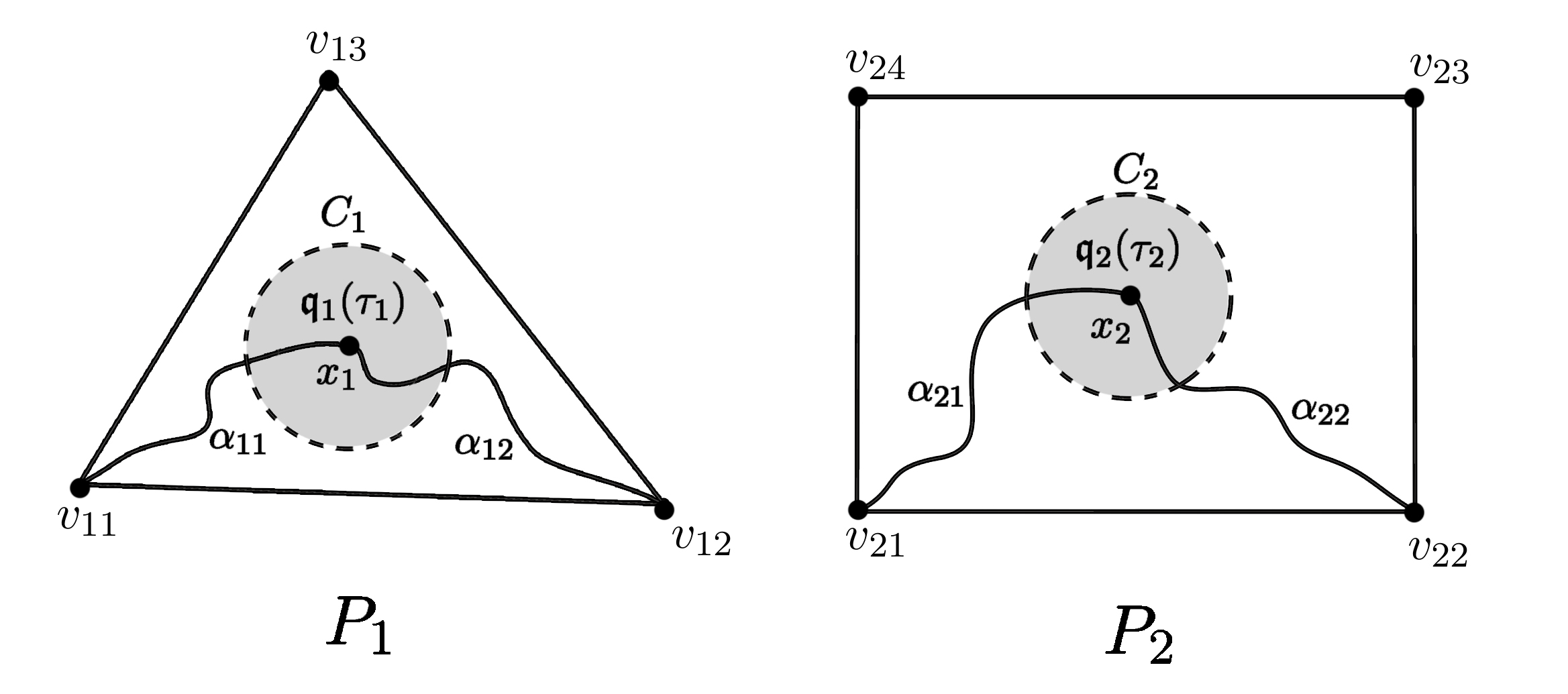}		%%%=====  Image
%\caption{}
%\label{}
\end{figure}
%%---------------------------
So
\begin{itemize}
\item $V_{11} = P_1 - \big\{ \; \mathfrak{q}_1(\tau_1) \cup [v_{12}, v_{13}] \cup \alpha_{12}(I^0) \big\} .$
\item $V_{12} = P_1 - \big\{ \; \mathfrak{q}_1(\tau_1) \cup [v_{11}, v_{13}] \cup \alpha_{11}(I^0) \big\} .$
\item $V_{13} = P_1 - \big\{ \; \mathfrak{q}_1(\tau_1) \cup [v_{11}, v_{12}] \cup \alpha_{11}(I^0) \big\} .$
\item $V_{21} = P_2 - \big\{ \; \mathfrak{q}_2(\tau_2) \cup [v_{22}, v_{23}] \cup [v_{23}, v_{24}]\cup \alpha_{22}(I^0) \big\} .$
\item $V_{22} = P_2 - \big\{ \; \mathfrak{q}_2(\tau_2) \cup [v_{23}, v_{24}] \cup [v_{21}, v_{24}]\cup \alpha_{21}(I^0) \big\} .$
\item $V_{23} = P_2 - \big\{ \; \mathfrak{q}_2(\tau_2) \cup [v_{21}, v_{22}] \cup [v_{21}, v_{24}]\cup \alpha_{21}(I^0) \big\} .$
\item $V_{24} = P_2 - \big\{ \; \mathfrak{q}_2(\tau_2) \cup [v_{21}, v_{22}] \cup [v_{22}, v_{23}]\cup \alpha_{21}(I^0) \big\} .$
\end{itemize}
Here $[v_{ij}, v_{kl}]$ is the edge joining the vertices $v_{ij}$ and $ v_{kl}$.
Clearly $Y_{ij}= \mathfrak{q}_i^{-1}(V_{ij})$ is $\mathbb{T}^2$-invariant and 
equivariantly contractible to the fixed point $\mathfrak{q}_i^{-1}(v_{ij}).$ Note
$$M_1 \#_{\mathbb{T}^2} M_2 = Y_{11} \cup Y_{12} \cup Y_{13} \cup Y_{21} \cup
 \cdots \cup Y_{24}.$$ Thus $cat_{\mathbb{T}^2}(M_1 \#_{\mathbb{T}^2} M_2) =3 + 4=7.$
\end{example}

%==============
%==============
\begin{theorem}\label{por}
 Let $M$ and $N$ be two $2n$-dimensional quasitoric manifolds with $p$ and $q$ many fixed points respectively.
Then $cat_{\TT^n}(M \times N) =pq$, where $\TT^n$-action on $M\times N$ is diagonal.
\end{theorem}
\begin{proof}
We adhere the notations of Theorem \ref{cat-quasi}. First observe that the diagonal $\TT^n$-action on $M \times N$ has $pq$
many fixed points. By Corollary \ref{lbcat}, $$cat_{\TT^n}(M \times N) \geq pq.$$

Let $X_u$ and $Y_v$ be $\TT^n$-categorical
open subsets of $M$ and $N$ respectively (as constructed in Theorem \ref{cat-quasi}), where $u \in M^{\TT^n}$ and $v \in N^{\TT^n}$. Let
$$H: X_u \times I \to X_u \quad \mbox{and} \quad K: Y_v \times I \to Y_v$$ be the respective $\TT^n$-homotopy such that
$$H(x, 0)=x, H(x, 1)=u, ~ \forall x \in X_u \quad \mbox{and} \quad K(y, 0)=y, K(y, 1)=v, ~ \forall y \in Y_v.$$
Then the $\TT^n$-homotopy $$L: X_u \times Y_v \times I \to X_u \times Y_v~
\mbox{defined by} ~L(x,y,r)=(H(x,r), K(y,r))$$ implies that $X_u \times Y_v
\subset M \times N$ is $\TT^n$-categorical. Since $$M \times N = \bigcup_{u\in M^{\TT^n}, v \in N^{\TT^n}} X_u \times Y_v,$$
 $cat_{\TT^n}(M \times N) \leq pq.$ Thus $cat_{\TT^n}(M \times N) = pq.$
\end{proof}
%==============

%==============
\begin{corollary}
Let $M_i$ be a $2n$-dimensional quasitoric manifold with $p_i$ many fixed points for $i=1, \ldots, l$.
Then $cat_{\TT^n}(M_1 \times \cdots \times M_l) = p_1 \ldots p_l$, where $\TT^n$-acts on
 $M_1 \times \cdots \times M_l$ diagonally. 
\end{corollary}

%==============
\begin{theorem}\label{qtm4}
Let $M$ be a 4-dimensional locally standard torus manifold over $P$, and $s $ be the number
of circles in $\partial P$ (see proof of Theorem 4.5). Then $ \big| M^{\TT^2} \big| + 2s \leq 
cat_{\TT^2}M \leq \big| M^{\TT^2} \big| + 2(s+1)$.
\end{theorem}
\begin{proof}
By Corollary \ref{lbcat}
$$ cat_{\TT^2}(M) \geq \Big|M^{\TT^2}\Big|.$$

Let $\mathfrak{q}: M \to P$ be the orbit map, and
$$X=\mathfrak{q}^{-1}(\bigcup_{i=1}^s C_i)=\bigcup_{i=1}^s\mathfrak{q}^{-1} (C_i),$$
 where $C_1, \ldots, C_s$ are the circles in $\partial{P}$.
We claim that if a $\TT^2$-categorical open subset $U$
contains a fixed point, then $U \cap X = \emptyset$. Suppose there is $z \in U \cap X$ and
$U$ contains a fixed point $v$. So $\oO(z) \subset U$. Since $ z \in X$, $\mathfrak{q}(z) \in C_i$ for some $i \in
\{1, \ldots, s\}$. Since $\TT^2$-action on $M$ is locally standard and $C_i \subset \partial{P}$,
$\oO(z)$ is homeomorphic to a circle and isotropy of $z$ is a circle subgroup of $\TT^2$.

Suppose $H: \oO(z) \times I \to M$ be a $\TT^2$-path from $\oO(z)$ to $\oO(v)=v$.
Then $\mathfrak{q} \circ H : z \times I \to P$ is a path from $\mathfrak{q}(z)$ to $\mathfrak{q}(v)$. Observe
that $Im(\mathfrak{q} \circ H) \cap P^0 \neq \emptyset $. Since isotropy group
over the interior $P^0$ is trivial, it is a contradiction to Lemma \ref{ht}.
This proves our claim.

On the other hand for each $i \in \{1, \cdots , s \}$, $\mathfrak{q}^{-1}(C_i)$ is homeomorphic to $C_i \times \zS^1$,
for some circle subgroup $\zS^1$ of $\TT^2$. Also for all $y \in \mf{q}^{-1}(C_i), \; \TT^2_y \cong \zS^1$.
Since $\TT^2$-action on $M$ is locally standard, there exists an equivariant tubular neighborhood $N_i$ of $\mf{q}^{-1}(C_i)$ such that $\TT^2_x$ is trivial for all $ x \in N_i - \mf{q}^{-1}(C_i)$. So by Lemma \ref{ht}, there is no $G$-path from an orbit in $\mathfrak{q}^{-1}(C_i)$ to any orbit in $M-\mf{q}^{-1}(C_i)$, and therefore $\mf{q}^{-1}(C_i)$ cannot be covered by a $\TT^2$-categorical open set. 

Suppose $U$ is a $\TT^2$-categorical subset such that 
$$ U \cap \mf{q}^{-1}(C_i) \neq \emptyset \neq U \cap \mf{q}^{-1}(C_j), \quad \text{for some } i \neq j.$$
So $U$ is $G$-homotopic to an orbit $\oO(z)$ in $M$. Therefore there exist a $G$-path from an orbit in $\mf{q}^{-1}(C_i)$ to $\oO(z)$, meaning $\oO(z) \subset \mf{q}^{-1}(C_i)$. Similarly $\oO(z) \subset \mf{q}^{-1}(C_j)$ which is a contradiction because $\mf{q}^{-1}(C_i)$ and $\mf{q}^{-1}(C_i)$ are disjoint
by locally standardness of the action. 

~\\
Hence 
$$ \Big| M^{\TT^2} \Big| + 2s \leq cat_{\TT^2}(M).$$

Let $Q_1, \ldots, Q_k$ be the edges
of $P$. To prove the other inequality, we adhere the notations of the proof of Theorem \ref{catM3}.
Since the fixed point set corresponds bijectively to the vertex set of $P$, we may assume
$\Big|M^{\TT^2} \Big| = x_1, \ldots, x_k$ where $k < l$.  Now choose an orientation on $P$ such that the vertex $x_i$ is the initial vertex of $Q_i$.
We denote the open cover of $P$ constructed in the proof of Theorem \ref{catM3} by $\mathcal{U}(P)$.
Let 
$$ R_i=\{U \in \mathcal{U}(P): U \cap Q_i \neq \emptyset ~\mbox{and}~ (V(Q_i) - \{x_i\}) \notin U\} ~\mbox{and}~ \mathcal{R}_i= \bigcup_{U \in R_i} U.$$
For simplicity, we may assume $x_{k+j}, x_{k+s+j} \in C_{j}$ for $j=1, \ldots, s$. Let
$$ R_{k+j}=\{U \in \mathcal{U}(P): U \cap C_j \neq \emptyset ~\mbox{and}~ x_{k+j} \notin U\}
~\mbox{and}~ \mathcal{R}_{k+j}= \bigcup_{U \in R_{k+j}} U.$$
It is an easy exercise to show that there is a codimension function preserving homeomorphism from $\mathcal{R}_i$
to $\RR^2_{\geq 0}$ if $1 \leq i \leq k$ and from $\mathcal{R}_{k+j}$ to $\RR \times \RR_{\geq 0}$ if
$1 \leq j \leq s$. So $\mf{q}^{-1}(\mathcal{R}_i)$ is equivariantly contractible to the orbit
$\mf{q}^{-1}(x_i)$ for $i=1, \ldots, k+s$. Also $\mf{q}^{-1}(X_{k+s+j})$ is equivariantly contractible
to $\mf{q}^{-1}(x_{k+s+j})$ for $j=1, \ldots, s$.

Let 
$$\mathcal{R}_{k+2s+1}=\bigcup_{y_i \in P^0}Y_i ~ \mbox{and}~ \mathcal{R}_{k+2s+2}=\bigcup_{z_j \in P^0}Z_j.$$
Recall that $Y_i$ and $Z_j$ are homeomorphic to open disc and subset of $ P^0$ if $y_i \in P^0$.
So $\mf{q}^{-1}(Y_i)$ and $\mf{q}^{-1}(Z_j)$ are equivariantly contractible to $\mf{q}^{-1}(y_i)$ and
$\mf{q}^{-1}(z_j)$ respectively. Since $Y_{i_1} \cap Y_{i_2} = Z_{j_1} \cap Z_{j_2} = \emptyset$ for 
$i_1 \neq i_2, j_1 \neq j_2$ and $P^0$ is path connected space, $\mf{q}^{-1}(\mathcal{R}_{k+2s+1})$ and 
 $\mf{q}^{-1}(\mathcal{R}_{k+2s+2})$ are equivariantly contractible to an orbit. Note that
 $$M= \bigcup_{i=1}^{k+s} \mf{q}^{-1}(\mathcal{R}_i) \cup \bigcup_{j=1}^s \mf{q}^{-1}(X_{k+s+j}) \cup
 \mf{q}^{-1}(\mathcal{R}_{k+2s+1}) \cup \mf{q}^{-1}(\mathcal{R}_{k+2s+2}).$$
 Therefore $ cat_{\TT^2}M \leq \big| M^{\TT^2} \big| + 2(s+1) $.
 \end{proof}

%**********************************************************************************
%*************************         6- Examples         ****************************
%**********************************************************************************
\section{Examples}\label{sec6}

%==============
\begin{example}\label{e1}
Consider the natural $\TT^2$-action on 
$$\zS^3=\{(z_1, z_2) \in \CC^2 : |z_1|^2+|z_2|^2=1\} ,$$
which is defined by 
$$(t_1, t_2) \cdot (z_1, z_2) \to (t_1z_1, t_2z_2).$$ 
Since all the isotropy groups $\TT^2_x$ are trivial except for $x=(1,0)$ and $x=(0, 1)$, by Lemma \ref{ghp} the orbits $\oO(1,0)$ and $\oO(0, 1)$ can not belong to a same $\TT^2$-categorical subset of $\zS^3$ and therefore $ cat_{\TT^2}(\zS^3) \geq 2$.
Let 
$$U_1 = \zS^3 - \oO(1,0) \quad \text{and} \quad U_2=\zS^3-\oO(0,1).$$
Let $B^2$ be the open disk. Since $U_1$ and $U_2$ are equivariantly homeomorphic to $\zS^1 \times B^2$,
there are $\TT^2$-homotopies from $U_1$ and $U_2$ onto the orbits $\oO(0, 1)$ and $\oO(1, 0)$ respectively.
Thus $cat_{\TT^2}(\zS^3) =2.$ 

\end{example}
%==============
\begin{example}\label{e2}
Consider the natural $\TT^2$-action on
$$ \zS^5=\big\{(z_1, z_2, z_3) \in \CC^3 : \; |z_1|^2+|z_2|^2+|z_3|^2 =1 \big\},$$
which is defined by 
$$(t_1, t_2) \cdot (z_1, z_2, z_3) \to (t_1z_1, t_2z_2, z_3).$$ 
An orbit of this action is either a point, circle, or torus; And $\zS^5$ is not contractible to any of them. So $ cat_{\TT^2}(\zS^5) \geq 2$.
Let 
$$ V_1 = \zS^5 -\{(0,0, -1)\} \quad \text{and} \quad V_2=\zS^5-\{(0,0,1)\}.$$
Clearly $V_1$ and $V_2$ are equivariantly contractible to the fixed points $(0,0,1)$ and $(0,0,-1)$ respectively. So they make a $\TT^2$-categorical covering of $\zS^5$. 
Thus $cat_{\TT^2}(\zS^5) =2.$ 
\end{example}
%==============
\begin{lemma}
Consider the $\TT^2$-actions defined in the Examples \ref{e1} and \ref{e2}. For any subgroup $H$ of $\TT^2$, the fixed point sets $(\zS^3)^H$ and $(\zS^5)^H$ are path-connected. Hence $\zS^3$ and $\zS^5$ are $\TT^2$-connected.
\end{lemma}
\begin{proof}
If $H=\{(1,1)\}$ is the trivial subgroup of $\TT^2$, then $(\zS^3)^H = \zS^3$, and it is path-connected. 
\begin{itemize}
\item Assume $H$ is non-trivial and there exist $\alpha \neq 1 \neq \beta$ such that ${p_0 = (\alpha,\beta) \in H}$. In this case
$$ (\zS^3)^H \subset (\zS^3)^{\{p_0\}} =\emptyset.$$
\item Assume $H$ is non-trivial and for all elements $(\alpha,\beta)$ in $H$, either $\alpha =1 $ or $\beta =1$.
If all elements of $H$ look like $(1,\beta)$, then
$$ (\zS^3)^H =\Big\{ (z_1,0) \in \zS^3 \; : \; \abs{z_1}^2 =1  \Big\} \cong \zS^1 .$$
Similarly if all elements of $H$ look like $(\alpha,1)$, then $ (\zS^3)^H \cong \zS^1 $.
\end{itemize}
Thus in any case $(\zS^3)^H$ is path-connected. Similarly one can show that $(\zS^5)^H$ is path-connected.
\end{proof}
%==============
%~\\

Note that every compact metric space is completely normal, so by Lemma \ref{GcomNor}, $\zS^3$, $\zS^5$ and $\zS^3 \times \zS^5$ are $\TT^2$-completely normal spaces.
%~\\

%==============
\begin{example}[Counterexample to Theorem 3.15 in \cite{CG}]\label{e4}
We adhere notations of Examples \ref{e1} and \ref{e2}. Let $X=\zS^3 \times \zS^5$. Consider the diagonal $\TT^2$-action on $X$, which is defined by 
$$t \cdot (p, q) \to (t \cdot p, t \cdot q).$$
Let $ A_0 =\emptyset, A_1 = U_1, A_2 =\zS^3$ and $ B_0 =\emptyset, B_1 = V_1, B_2 =\zS^5$. Clearly 
 $A_0 \subset A_1 \subset A_2$ and $B_0 \subset B_1 \subset B_2$ are $\TT^2$-categorical sequences for $\zS^3$ and $\zS^5$ respectively. Consider the sequence 
\begin{equation}\tag{$\star$}\label{seq}
 C_0 \subset C_1 \subset C_2 \subset C_3 
\end{equation}
where 
$$ C_0 =\emptyset, \quad C_1 = A_1 \times B_1, \quad C_2=A_2 \times B_1 \cup A_1 \times B_2 ,\quad\mbox{and}\quad C_3=A_2 \times B_2 = X.$$ 
Although $\zS^3$, $\zS^5$ and $X$ satisfy the conditions in Theorem 3.15 in \cite{CG}, we show that 
$$C_2 - C_1 = (A_2 -A_1) \times B_1 \cup A_1 \times (B_2 - B_1)$$ 
does not sit in any $\TT^2$-categorical set of $X$, and therefore \eqref{seq} is not a $\TT^2$-categorical sequence. 

Let $\zS_1^1$ and $\zS_2^1$ be the circle subgroups of $\TT^2$ determined by the standard vectors $e_1$ and $e_2$ in $\ZZ^2$ respectively. Let $x = ((1,0),(0,0,1))$ and $y = ((0, 1),(0,0,-1))$. Note that 
 $$\oO(x) \subset(A_2-A_1) \times B_1 \quad\mbox{and}\quad \oO(y) \subset A_1 \times (B_2 -B_1) .$$ 
Also for isotropy groups we have, $ \TT^2_x = \zS_2^1$ and $\TT^2_y =\zS_1^1$. Suppose there
exists $z \in X$ with $\TT^2$-paths from
$\oO(x)$ to $\oO(z)$ and from $\oO(y)$ to $\oO(z)$. By Lemma \ref{ht}, $\zS_1^1$ and $\zS_2^1$ are subgroups
 of $\TT^2_z$.
Thus $z$ is a fixed point. But $\TT^2$-action on $X$ has no fixed point, therefore by Lemma \ref{ghp} there is no $\TT^2$-categorical subset in $X$ containing $C_2-C_1$. This contradicts the arguments in the proof of Theorem 3.15 in \cite{CG}. 

Here we show that $ cat_{\TT^2}(\zS^3 \times \zS^5) = 4 $. Clearly $ U_1 \times V_1 $, $ U_1 \times V_2 $, $ U_2 \times V_1$, and $U_2 \times V_2$ form a $\TT^2$-categorical cover for  $\zS^3 \times \zS^5$. Hence $ cat_{\TT^2}(\zS^3 \times \zS^5) \leq  4  $. On the other hand according to orbit types of $\TT^2$-action on $\zS^3 \times \zS^5$, one can show that the isotropy groups are either trivial or homeomorphic to $\zS^1$. So by using Theorem 3.7 in \cite{HT}, it is enough to show that
$$ cat_{\TT^2}(\zS^1 \times \zS^3) \geq  2.  $$
By looking at homology groups, it is clear that $\zS^1 \times \zS^3$ cannot contract to an orbit. Hence $ cat_{\TT^2}(\zS^1 \times \zS^3) $ cannot be one. Thus
$$ cat_{\TT^2}(\zS^3 \times \zS^5) \geq  cat_{\TT^2}(\zS^1 \times \zS^3) +  cat_{\TT^2}(\zS^1 \times \zS^3) \geq 4 .$$

\end{example}

%==============
\begin{example}[Counterexample to Theorem 3.16 in \cite{CG}]\label{e5}
Let $M$ and $N$ be $2m$ and $2n$ dimensional quasitoric manifolds over the polytopes $P$ and $Q$
respectively. Then $M \times N$ is a $4mn$-dimensional quasitoric manifold over $P \times Q$.
By Theorem \ref{cat-quasi},
$$ cat_{\TT^m \times \TT^n}(M  \times N)= \big| V(P \times Q) \big| = \big| V(P) \big| \times  \big| V(Q) \big| = cat_{\TT^m}(M) \times cat_{\TT^n}(N).$$

Note that $M$ is a $\TT^m$-manifold, $N$ is a $\TT^n$-manifold, and $M \times N$ is a $\TT^m\times \TT^n$-manifold. Also $M \times N$ is a compact metrizable space, so it is completely normal.
\end{example}
%==============
\begin{example}\label{e3}
We adhere the notation of Example \ref{33}. Let 
$$ V_1 = \zS^{2n} -\{(0,\cdots,0, -1)\} \quad , \quad V_2=\zS^{2n}-\{(0,\cdots,0,1)\}.$$
Since $V_1$ and $V_2$ are equivariantly contractible to the fixed points $(0, \cdots, 0,1)$ and $(0,\cdots,0,-1)$ respectively,
so they are $\TT^n$-categorical subset of $\zS^{2n}$. 
Thus $cat_{\TT^n}(\zS^{2n}) =2.$ In particular $cat(S^{2n})=2$, since $\zS^{2n}$ is not contractible.
\end{example}
%==============

%==============
\begin{example}
Let $p>0, q_1,\ldots, q_n$ be integers such that $p$ and $q_i$ are relatively prime for all $i=1,\ldots, n$.
Consider $$\zS^{2n+1}=\{(z_1,\ldots, z_{n+1})\in \mathbb C^{n+1} :\; |z_1|^2+\cdots + |z_{n+1}|^2=1\}.$$
The {\em  $(2n+1)$-dimensional lens space}  $L = L(p;q_1,\ldots, q_n)$  is the orbit space $\zS^{2n+1}/\mathbb Z_p$ where $\mathbb Z_p$-action on $\zS^{2n+1}$  is defined by
$$\theta\colon \mathbb Z_p\times \zS^{2n+1}\to \zS^{2n+1},$$
$$([k], (z_1,\ldots, z_{n+1}))\mapsto (e^{2kq_1\pi \sqrt{-1}/p}z_1, \ldots, e^{2kq_n\pi \sqrt{-1}/p}z_n, e^{2k\pi \sqrt{-1}/p}z_{n+1}).$$
The equivalence class of $(z_1, \ldots, z_{n+1})$ is denoted by $[z_1, \ldots, z_{n+1}]$. The $(n+1)$-dimensional
compact torus $\TT^{n+1}$ acts on $L$ by:
\begin{equation}\label{tlens}
 (t_1, \ldots, t_{n+1}) \times [z_1, \ldots, z_{n+1}] \to [t_1z_1, \ldots, t_{n+1}z_{n+1}].
\end{equation}
Let $e_1, \ldots, e_{n+1}$ be the standard vectors in $\CC^{n+1}$, and $[e_i]$ be the equivalence class of $e_i$
in $L$. The orbit of $[e_i]$ is $\oO_i = \{[0, \ldots, 0, z_i, 0, \ldots, 0] : |z_i|=1\}$.
From the action in Equation (\ref{tlens}) $\oO_1, \ldots, \oO_{n+1}$ are the only orbits of dimension one
and there is no orbit of dimension less than one.  Suppose there are $\TT^{n+1}$-paths from
$\oO_i$ to $\oO(z)$ and from $\oO_j$ to $\oO(z)$ for some $z \in L $ with $i \neq j$.
So we get inclusions of isotropy groups,
$$\TT^{n+1}_{e_i} \subset \TT^{n+1}_z \quad \mbox{and} \quad \TT^{n+1}_{e_j} \subset \TT^{n+1}_z.$$
Thus $\TT^{n+1}_z =\TT^{n+1}$, since $i \neq j$. This contradicts the fact that $\TT^{n+1}$-action on
$L$ has no fixed point.
By Lemma \ref{ghp}, $\oO_i$ and $\oO_j$ can not
belong to same $\TT^{n+1}$-categorical subset of $L$. Thus $$cat_{\TT^{n+1}}(L) \geq n+1.$$

Let 
$$U_i=\{[z_1, \ldots, z_{n+1}] \in L : z_i \neq 0\}, \quad \text{ for }\; i=1, \ldots, n+1.$$
Then $U_i$ is invariant open subset of $L$. It is not difficult to show that
$U_i$ is a $\TT^{n+1}$-categorical set containing $\oO_i$. Since $U_1,
\ldots, U_{n+1}$ covers $L$, $cat_{\TT^{n+1}}(L) \leq n+1$.
Hence $$cat_{\TT^{n+1}}(L)= n+1.$$ 
\end{example}
%==============

{\bf Acknowledgement.} The authors would like to thank Professor Donald Stanley and the reviewers for the very
 helpful comments and suggestions. 
The second author thanks Pacific Institute for Mathematical Sciences and University of Regina for financial support.

%%%%%%%%%%%%%%%%%%%%%%%%%%%%%%%%%%%%%%%%%%%%%%%%%%%%%%%%%%%
%%%%%%%%%%%%%%%      Bibliography       %%%%%%%%%%%%%%%%%%%
%%%%%%%%%%%%%%%%%%%%%%%%%%%%%%%%%%%%%%%%%%%%%%%%%%%%%%%%%%%

\bibliographystyle{abbrv}

\bibliography{Myref}

\end{document}